\documentclass[11pt]{article}

\usepackage{amsmath}
\usepackage{amsthm}
\usepackage{amsfonts}
\usepackage{setspace}
\usepackage{fullpage}
\usepackage{amssymb}
\usepackage{enumitem}
\usepackage{bbold} 
\usepackage{comment}

\usepackage{authblk}
\usepackage{hyperref}

\newtheorem{lemma}{Lemma}
\newtheorem{theorem}{Theorem} 
\newtheorem{proposition}{Proposition}
\newtheorem{corollary}{Corollary}

\newcommand{\A}{\mathcal{A}}

\newcommand{\C}{\mathcal{C}}

\newcommand{\T}{\mathcal{T}}
\newcommand{\s}{\mathcal{S}}
\newcommand{\abs}[1]{\left\lvert{#1}\right\rvert}
\newcommand{\floor}[1]{\left\lfloor{#1}\right\rfloor}

\newcommand{\ex}{{\rm  ex}}
\newcommand{\slfrac}[2]{\left.#1\middle/#2\right.}

\title{Localized versions of extremal problems}

\author[1]{David Malec}
\author[2]{Casey Tompkins}
\affil[1]{University of Maryland, College Park}
\affil[2]{Alfr\'ed R\'enyi Institute of Mathematics}

\begin{document}

\maketitle

\begin{abstract}
We generalize several classical theorems in extremal combinatorics by replacing a global constraint with an inequality which holds for all objects in a given class. In particular we obtain generalizations of Tur\'an's theorem, the Erd\H{o}s-Gallai theorem, the
LYM-inequality, the Erd\H{o}s-Ko-Rado theorem and the Erd\H{o}s-Szekeres theorem on sequences.
\end{abstract}

\section{Introduction}
In this paper we will consider a systematic way to generalize some classical problems in extremal combinatorics.  
In an extremal problem one typically seeks to maximize the number of some object among a class of finite structures (graphs, hypergraphs, sequences etc.) avoiding some other object.  
We seek to reformulate such problems in the following more general setting. 
Rather than considering structures of some class  forbidding a given object, we define a weight function which holds for arbitrary structures from that class.
We then prove a bound on this weight function that immediately implies the original extremal theorem. 

We begin by discussing the case of graphs. For a graph $G$, we denote the set of vertices and edges of $G$ by $V(G)$ and $E(G)$, respectively. 
A clique with $r$ vertices is denoted by $K_r$. 
For a given graph $F$, the Tur\'an number $\ex(n,F)$ is the maximum number of edges possible in an $n$-vertex graph which does not contain $F$ as a subgraph. 

Tur\'an's theorem asserts that for all $r$ we have $\ex(n,K_{r+1}) \le \frac{(r-1)n^2}{2r}$, and equality holds when $r$ divides $n$ and we take a complete $r$-partite graph with classes of equal size.
Now we attempt to generalize this result by considering an arbitrary $n$-vertex graph $G$ and looking at the behavior of the edges.
Since for any $r$ we wish to recover the bound of $\frac{(r-1)n^2}{2r}$ in the case when no edge is in a copy of $K_{r+1}$, we define a weight on the edges based on the size of the cliques in which that edge occurs.
To this end, let 
\[
c(e) = \max\{r:\mbox{$e$ occurs in a subgraph of $G$ isomorphic to $K_r$}\}.
\]
It is then natural to assign each edge a weight $\slfrac{1}{\frac{(c(e)-1)n^2}{2c(e)}}$ with denominator matching the corresponding extremal bound. 
We wish to prove that the sum of these weights is at most $1$.
Moving the $n^2/2$ factor to the other side of the inequality we arrive at the following natural assertion.
\begin{theorem} \label{localizedTuran}
Let $G$ be any $n$-vertex graph, then
\[
\sum_{e \in E(G)} \frac{c(e)}{c(e)-1} \le \frac{n^2}{2},
\]
and equality holds if and only if $G$ is a multipartite graph with classes of equal size.
\end{theorem}
We prove Theorem~\ref{localizedTuran} in Section~\ref{graphs} using an inductive argument.
Theorem~\ref{localizedTuran} was very recently proved independently by Domagoj Brada\u{c}~\cite{bradac} using the method of Motzkin and Straus~\cite{motzkin1965maxima}.
Brada\u{c}~\cite{bradac} communicates that the statement of Theorem~\ref{localizedTuran} was proposed by Balogh and Lidick\'y and discussed at a conference in Oberwolfach in 2022.
Moreover he mentions that stability results in the case of $K_5$-free graphs have been obtained by Balogh and Lidick\'y.
We will refer to Theorem~\ref{localizedTuran} as a localized version of Tur\'an's theorem since the weight assigned to each edge depends only on the structures which it participates in.
The bound from Tur\'an's theorem for any $r$ is easily recovered by noting that in a $K_{r+1}$-free graph each edge has weight at most $\frac{r}{r-1}$.

Next we turn our attention to the case of paths. 
Erd\H{o}s and Gallai~\cite{gallai1959maximal} proved that an $n$-vertex graph without a copy of a path of length $k$ contains at most $\frac{(k-1)n}{2}$ edges. Let $P_k$ denote the path of length $k$ (that is, with $k$ edges).  
Again we introduce a function on the edges 
\[
p(e) = \max\{k:\mbox{$e$ occurs in a subgraph of $G$ isomorphic to $P_k$}\}.
\]
Then we have the following localized version of the theorem of Erd\H{o}s and Gallai.
\begin{theorem} \label{localizedErdosGallai}
Let $G$ be any $n$-vertex graph, then
\[
\sum_{e\in E(G)} \frac{1}{p(e)} \le \frac{n}{2},
\]
and equality holds if an only if every connected component of $G$ is a clique.
\end{theorem}
Again, the original theorem of Erd\H{o}s and Gallai is easily recovered by taking a $P_k$-free graph~$G$ and using that $p(e) \le k-1$.
Also one can rather simply give a similar result for stars. 
For any graph $G$, let $s(e)$ be the maximum number of edges in a star containing $e$.
\begin{proposition} \label{stars}
Let $G$ be any $n$-vertex graph, then
\[
\sum_{e\in E(G)} \frac{1}{s(e)} \le \frac{n}{2},
\]
and equality holds if and only if each connected component of $G$ is regular.
\end{proposition}

Abstractly we are looking for theorems of the following form.
Suppose we have an infinite sequence of graphs $F_1,F_2,\dots$ with the property that for all $i$, $F_i$ is a subgraph of $F_{i+1}$.
Let $G$ be an arbitrary $n$-vertex graph, and for each edge $e \in E(G)$ set 
\[f(e) = \max\{i:\mbox{$e$ is an edge of a subgraph of $G$ isomorphic to $F_i$}\}.\]
Then we would like to know when it holds that 
\begin{equation} \label{weighted}
\sum_{e\in E(G)} \frac{1}{\ex(n,F_{f(e)+1})} \le 1.
\end{equation}
A bound of this form simultaneously generalizes each of the corresponding results for extremal numbers~$\ex(n,F_i)$. 
Indeed, if $G$ does does not contain a particular $F=F_i$ as a subgraph, then by the monotonicity of the extremal function $\ex(n,F)$ with respect to taking subgraphs, we have
\[
\frac{\abs{E(G)}}{\ex(n,F)} \le \sum_{e\in E(G)} \frac{1}{\ex(n,F_{f(e)+1})} \le 1,
\]
and consequently $\abs{E(G)} \le \ex(n,F)$. 
A stronger corollary of a bound of the form~\eqref{weighted} is the following. Suppose for an infinite sequence of graphs $F_1\subset F_2\subset\cdots$ we have a bound of the form~\eqref{weighted}. 
Fix $F=F_i$ for some $i$, and partition the edge set $E(G)$ into two parts $S$ and $T$ where $S = \{e\in E(G): \mbox{$e$ is an edge of a subgraph of $G$ isomorphic to $F$}\}$ and $T = E(G) \setminus S$, then 
\begin{equation} \label{borgtype}
\frac{\abs{S}}{\ex(n,F)} + \frac{\abs{T}}{\binom{n}{2}} \le 1.
\end{equation}
The bound~\eqref{borgtype} follows by replacing $\ex(n,F_j)$ with $\ex(n,F_i)$ for all $j<i$ and replacing $\ex(n,F_j)$ with $\binom{n}{2}$ for all $j>i$. 
Note that Theorems~\ref{localizedTuran} and~\ref{localizedErdosGallai} are not precisely of the form~\eqref{weighted} since the corresponding extremal bounds which we generalize are only tight in certain divisibility cases. 
It would be interesting to extend these results so that the inequality is sharp in all divisibility cases.  Note that inequalities of the form~\eqref{weighted} do not hold for any possible graph sequence.  Indeed, if $F_1$ is a triangle, $F_2$ is a bow tie and $G$ is a balanced complete bipartite graph plus an edge, then~\eqref{weighted} is violated.

In extremal set theory bounds of the  form~\eqref{borgtype} sometimes have had applications for `cross'-versions~\cite{hilton1977intersection} (also termed `multicolor'-versions~\cite{bollobas2004multicoloured,keevash2004multicolour}) of various problems.
We now turn our attention to finding localized versions of some theorems in extremal set theory.

First we recall the classical LYM-inequality~\cite{bollobas1965generalized,lubell1966short,meshalkin1963generalization,yamamoto1954logarithmic} which generalizes Sperner's theorem~\cite{sperner1928satz} about antichains in the Boolean lattice. 
Let $\A \subset 2^{[n]}$ be a family of sets forming an antichain with respect to the subset relation, then we have
\begin{equation} \label{LYM}
\sum_{A\in\A} \frac{1}{\binom{n}{\abs{A}}}\le 1.  
\end{equation}
Daykin and Frankl proved the following natural extension of the LYM-inequality.  Suppose that $\A$ is any family of subsets of $[n]$ and let $\s \subset \A$ be the collection of those sets in $\A$ incomparable to every set in $\A$. 
Write $\T = \A \setminus \s$, then
\begin{equation}\label{daykin}
\sum_{A\in\s} \frac{1}{\binom{n}{\abs{A}}} + \frac{\abs{\T}}{2^n} \le 1.
\end{equation}
The inequality~\eqref{daykin} can be used to directly bound the sum of the sizes of a family of cross-Sperner families. 

A $(k+1)$-chain in $2^{[n]}$ is a collection of sets $A_1,A_2,\dots,A_{k+1}$ such that for all $i$ we have $A_i \subset A_{i+1}$. Another natural extension of the LYM-inequality is the following. 
Suppose $\A \subset 2^{[n]}$ contains no $(k+1)$-chain, then
\begin{equation}\label{katona}
\sum_{A\in\A} \frac{1}{\binom{n}{\abs{A}}}\le k.
\end{equation}
The inequality~\eqref{katona} can be proved by applying a Mirsky-type~\cite{mirsky1971dual} decomposition to the set family or directly by imitating Lubell's proof of the LYM-inequality~\cite{katona}.  
Our main result in this topic is a generalization of~\eqref{katona}.
For an arbitrary $\A \subset 2^{[n]}$, let 
\[
c(A) = \max \{k:\mbox{$A$ participates in a $k$-chain consisting of sets from $\A$}\}.
\]
\begin{theorem}\label{localizedLYM} Let $\A\subset 2^{[n]}$ be an arbitrary family of sets, then
\[
\sum_{A\in \A} \frac{1}{\binom{n}{\abs{A}}c(A)}\le 1.
\]
% Moreover equality holds for families which consist of the union of complete levels.
Moreover equality holds if and only if $\A$ consists of a union of complete levels.
\end{theorem}
A further extension of Theorem~\ref{localizedLYM} to the setting of posets is given in Section~\ref{posets}, Corollary~\ref{posetcA}. Observe that if $\A$ is a $(k+1)$-chain free family, then $c(A) \le k$ and we recover~\eqref{katona}.
The inequality in Theorem~\ref{localizedLYM} and~\eqref{daykin} seem to be independent of each other.

Next we consider intersecting set families. Recall that the Erd\H{o}s-Ko-Rado~\cite{erdos1961intersection} theorem asserts than if $r<n/2$, then a pairwise intersecting family consisting of sets of size $r$ has size at most~$\binom{n-1}{r-1}$. 
In order to find a shorter proof for a theorem of Hilton~\cite{hilton1977intersection}, Borg~\cite{borg2009short} proved the following generalization of the Erd\H{o}s-Ko-Rado theorem.
Let $r<n/2$ and let $\A$ be an arbitrary family of $r$-element subsets of $[n]$.  Let $\s = \{A\in\A:\mbox{for all $B\in\A$, $A$ and $B$ intersect}\}$ and $\T=\A\setminus \s$, then
\begin{equation}
\abs{\s} + \frac{r}{n}\abs{\T} \le \binom{n-1}{r-1}.
\end{equation}
%\[
%\frac{\abs{\s}}{\binom{n-1}{r-1}} + \frac{\abs{\T}}{\binom{n}{r}} \le 1.
%\]
We now give a localized version generalizing Borg's result. 
\begin{theorem}
\label{m(A)}
Let $\A$ be a collection of $r$-element subsets of $[n]$. Let $m$ be the function on $\A$ defined by setting $m(A)=n/r$, if $A$ is contained in some matching of $\floor{n/r}$ sets from $\A$, and setting $m(A)$ equal to the size of the largest matching in $\A$ containing $A$, if that matching has fewer than $\floor{n/r}$ sets, then
\begin{displaymath}
\sum_{A\in\A} \frac{1}{m(A)} \le \binom{n-1}{r-1}.
\end{displaymath}
\end{theorem}
Note that there are at least three distinct constructions attaining equality in Theorem~\ref{m(A)}.
In addition to a star or a full level we may take the construction consisting of all sets of size $r$ containing at least one of two fixed elements.

Finally, we turn out attention to perfect graphs, posets and sequences. Recall that a perfect graph is one in which the chromatic number and clique number of all induced subgraphs are equal.
In particular for a perfect graph $G$ we have $\omega(G)= \chi(G)$ and so $\alpha(G)\omega(G)=\alpha(G)\chi(G)\ge n$.
In Section~\ref{posets} we prove the analogous localized version of this bound as a corollary of a more technical statement about functions on perfect graphs that may be of independent interest.

Let $G$ be a perfect graph, and for each $v\in V(G)$ let $c(v)$ be the largest size of a clique containing~$v$, and let $i(v)$ be the largest size of an independent set containing $v$.

\begin{theorem}\label{localizedperfect}
For any $n$-vertex perfect graph $G$, we have
\[
\sum_{v\in V(G)} \frac{1}{c(v)i(v)} \le 1.  
\]
\end{theorem}
As a consequence of Theorem~\ref{localizedperfect} we deduce a localized version of the classical theorem of Erd\H{o}s and Szekeres~\cite{erdos1935combinatorial} about sequences.

\section{Proofs of localized versions of extremal graph theorems} \label{graphs}

% %not sure if we need to spell out this notation.
%  For graphs $G$ and $H$ we use $G\setminus H$ to denote $G[V(G)\setminus V(H)]$.  For parameters such as $c(e)$ defined on a graph $G$ and a subgraph $H$ of $G$, we use a subscript $c_H(e)$ to denote the parameter defined in terms of the graph

  \begin{proof}[Proof of Theorem~\ref{localizedTuran}]
  We prove the claim by induction on $n$. The result holds
  trivially for $n=1$, so assume $n>1$ and let $G$ be an $n$-vertex graph with edge set $E$. If the graph contains no edges, then the bound follows trivially so assume $G$ has at least one edge.  
  Let $k$ be the size of the
  largest clique in $G$, and let $C$ be a clique of size $k$. 
  We
  split the edges $E$ into three parts: $E_{C}$, the edges within $C$;
  $E_{G\setminus C}$, the edges within $G\setminus C$; and $E_{S}$,
  the edges that connect $C$ and $G\setminus C$.  
  We bound the
  contribution to the sum from each of these separately.
  \begin{enumerate}
  \item Since $C$ is a clique of maximum size in $G$, we know that
    $E_{C}$ contains $k(k-1)/2$ edges, all with $c(e)=k$.  So we can
    see that
    \begin{equation*}
      \sum_{e \in E_{C}}\frac{c(e)}{c(e)-1}
      =\frac{k(k-1)}{2}\frac{k}{k-1}
      =\frac{k^2}{2}.
    \end{equation*}
  \item For all $v \in V(G \setminus C$), let $C_v=\{w \in C \mid
    \{v,w\}\in E\}$.  
    Note that $C_v\cup\{v\}$ is itself a clique, and
    so we may conclude that $\abs{C_v}+1\le k$, and $c(e)\ge
    \abs{C_v}+1$ for all edges $e$ that connect $v$ to $C$.  
    Thus, we
    obtain that
    \begin{equation*}
      \sum_{e \in E_{S}}\frac{c(e)}{c(e)-1}
      \le \sum_{v\in V(G \setminus C)}\abs{C_v}\frac{\abs{C_v}+1}{\abs{C_v}}
      \le \sum_{v\in V(G \setminus C)} k
      =(n-k)k.
    \end{equation*}
  \item Lastly, our induction hypothesis implies that 
  \begin{equation*}
    \sum_{e \in E_{G \setminus C}}\frac{c(e)}{c(e)-1}
    \le \frac{(n-k)^2}{2}.
    \end{equation*}
  \end{enumerate}
  Combining all three of the estimates above we obtain that
  \begin{equation}\label{combine}
    \sum_{e \in E} \frac{c(e)}{c(e)-1}
    \le \frac{k^2}{2}+k(n-k)+\frac{(n-k)^2}{2}
    =\frac{n^2}{2},
  \end{equation}
  as required. 
  
  Now suppose we have equality in~\eqref{combine}, then all three of the estimates above must be tight. 
  From the third estimate we obtain that $G\setminus C$ is a complete multipartite graph with $m$ equal size classes where $m\le k$.
  Moreover, for each edge $e$ in $E_{G\setminus C}$ we have that $m=c_{G\setminus C}(e) = c_G(e)$. 
  By the second estimate each vertex from $G\setminus C$ must have exactly $k-1$ neighbors in $C$. 
  Consequently if we take a vertex from each class in $G\setminus C$, we obtain that these vertices have a common neighborhood in $C$ of size at least $k-m$. 
   % Thus every edge in $E_{G\setminus C}$ is part of a $k$-clique and so $m=k$.
  Thus every edge in $E_{G\setminus C}$ is part of a $k$-clique in $G$ and so $m=c_{G}(e)=k$.
  Finally, any selection of $k$ vertices from distinct classes of $G\setminus C$ must have an empty common neighborhood in $C$ for otherwise we would have a larger clique. 
  It follows that for every class in $G\setminus C$ there is exactly one vertex in $C$ to which all vertices in that class are nonadjacent, and this vertex is distinct for each class in $G\setminus C$. 
  Adding each such vertex from $C$ to its respective class in $G\setminus C$, we see that $G$ is a balanced multipartite graph as required.
\begin{comment}

\begin{corollary}
$\sum_{v \in V} d(v) \frac{c(v)}{c(v)-1} \le n^2$
\end{corollary}
\begin{proof}
\begin{align*}
\sum_{v \in V} d(v) \frac{c(v)}{c(v)-1} &= \sum_{v \in V} \sum_{e \in E: v \in e}  \frac{c(v)}{c(v)-1} \\
&= \sum_{e \in E} \sum_{v \in e} \frac{c(v)}{c(v)-1} \\
&= \sum_{uv=e \in E} \frac{c(u)}{c(u)-1} + \frac{c(v)}{c(v)-1} \\
&\le 2 \sum_{e \in E} \frac{c(e)}{c(e)-1} \\
&\le n^2.
\end{align*}
The first inequality follows from the fact that $v \in e$ implies $c(v) \ge c(e)$ and $x/(x-1)$ is a decreasing function for $x \ge 2$.  The second inequality follows from Theorem \ref{thm:main}.
\end{proof}

\end{comment}
\end{proof}

Now we give the proof of the localized version of the Erd\H{o}s-Gallai theorem.

\begin{proof}[Proof of Theorem~\ref{localizedErdosGallai}]
  We prove the claim by induction on $n$.  The claim follows trivially
  for $n=1$, so assume $n>1$.  Furthermore, assume that $G$ is
  connected; if $G$ is not connected, we may simply apply the
  induction hypothesis to each connected component in turn, and sum
  the results to arrive at our claim.  

  Assume the longest path in $G$ has length $k$, and let $P$ be a path
  achieving this length, with endpoints $v$ and $w$.  
  We use $R(v)$
  to denote the edges that have $v$ as an endpoint; note that $R(v)$
  cannot contain any edges to vertices outside of $P$, since otherwise
  $P$ is not maximal.  We define $R(w)$ similarly, and note it has
  the same property.

  If there exists a circuit $C$ that includes all the vertices on $P$,
  then we can see that $P$ must include all the vertices in $G$.
  Otherwise, since $G$ is connected, we can find a path from $C$ to
  any vertex $u$ not lying on it; but the  path that starts at $u$,
  goes to $C$, and then goes around $C$ is strictly longer that $P$, a
  contradiction.  So in this case, we have that $k=n-1$.  Furthermore,
  every edge in the graph either lies on the circuit $C$ or forms a chord in $C$. In either case, we can easily construct a path of length
  $n-1$ containing it.  So we can see that
  \begin{equation}\label{all}
    \sum_{e \in E(G)} \frac{1}{p(e)}
    = \abs{E(G)}\frac{1}{n-1}
    \le \binom{n}{2}\frac{1}{n-1}
    = \frac{n}{2},
  \end{equation}
  and we are done. Equality in the theorem implies~\eqref{all} is tight and so $G$ is a clique. 

  Now we assume that there is no circuit $C$ that includes all of the
  vertices of $P$.  
  Note that this means that $\{v,w\}\notin E(G)$.
  Furthermore, let $\{v',w'\}$ be any internal edge of $P$, where $v'$
  is closer to $v$ along $P$ (and so $w'$ is closer to $w$ along $P$).
  Note that $P\setminus\{\{v',w'\}\}\cup\{\{v,w'\},\{v',w\}\}$ would be a
  circuit, so either $\{v,w'\}\notin E(G)$ or $\{v',w\}\notin E(G)$.  Recalling
  that $R(v)$ and $R(w)$ contain no edges to vertices outside
  $P$, we may thus conclude that
  \begin{equation*}
    \abs{R(v)}+\abs{R(w)}\le 2 + (k-2) = k.
  \end{equation*}
  So without loss of generality, we assume that $\abs{R(v)}\le
  k/2$.  Note that any $e\in R(v)$ has $p(e)=k$, since there is
  always some $e'\in P$ that can be replaced with $e$ to form a new
  path $P'$.  Thus, if we apply our induction hypothesis to
  $G\setminus v$, we get that
  \begin{equation}\label{longinequality}
    \sum_{e \in E(G)}\frac{1}{p(e)}
    =\sum_{e \in R(v)}\frac{1}{p(e)}+\sum_{E(G)\setminus R(v)}\frac{1}{p(e)}
    \le \sum_{e \in R(v)}\frac{1}{p(e)}+\sum_{E(G)\setminus R(v)}\frac{1}{p_{G\setminus v}(e)}
    \le \frac{k}{2}\frac{1}{k}+\frac{n-1}{2}
    =\frac{n}{2}.
  \end{equation}
  
  Equality in the theorem statement implies that both inequalities in~\eqref{longinequality} hold with equality for all $e \in E(G)$. This means that for the edges $e$ in $E(G\setminus v)$ we have $p(e) = p_{G\setminus v}(e)$. Moreover by induction $G\setminus v$ must be a disjoint union of cliques.  If $k=1$ the graph consists of just an edge and we are done, so assume $k>1$.  Consider the edge $e$ incident to $v$ in $P$ and the next edge $f$ in the path.  Since $f$ belongs to a connected component which is a clique in $G\setminus v$ and $f$ is adjacent to $e$, we can easily find a longer path in $G$ containing $f$ and so $p(e)>p_{G\setminus v}(e)$. Thus in the present case, equality cannot hold in~\eqref{longinequality}. 
 \end{proof}
We conclude this section with the proof of the analogous statement about stars.  
\begin{proof}[Proof of Proposition~\ref{stars}]
For a vertex $v$ of $G$ set $w(v)$ equal to the sum of $\frac{1}{s(e)}$ across edges $e$ which are incident to $v$.  Then
for all $v$,
\begin{equation} \label{stareq}
w(v) \le \sum_{e:v\in e} \frac{1}{d(v)} = 1.
\end{equation}

Then we have
\[
2 \sum_{e\in E(G)} \frac{1}{s(e)} = \sum_v w(v) \le n,
\]
and the result follows.

Suppose we have equality in~\eqref{stareq}, then for every $v$ every edge $e$ incident to $v$ must satisfy $s(e)=d(v)$. It is then immediate that each component is regular, as claimed.
\end{proof}
Given that we have localized versions of the extremal results of paths and stars it would be natural to investigate what happens in the case of an arbitrary sequence of trees: $T_1 \subset T_2 \subset\cdots$.  However, such an investigation may be difficult since we only know the validity of the Erd\H{o}s-S\'os conjecture in certain cases.  Nonetheless one could hope to prove localized versions conditional on the Erd\H{o}s-S\'os conjecture.

\section{Proofs of localized versions of extremal hypergraph theorems} \label{hypergraphs}

We begin by proving the localized version of the LYM-inequality.

\begin{proof}[Proof of Theorem~\ref{localizedLYM}]
Let $\A \subset 2^{[n]}$ be any family of sets.  Call a chain of sets $\C$ maximal if it contains a set of every possible cardinality.  We will double count pairs $(A,\C)$ where $A \in \A$ and $\C$ is a maximal chain in $2^{[n]}$ using a weight function $w$.  We define $w$ by
\begin{displaymath}
w(A,\C) = 
\begin{cases}
\frac{1}{c(A)} & \mbox{if } A\in \A  \mbox{ and } A \in \C, \\
0 & \mbox{otherwise.}
\end{cases}
\end{displaymath}
First fix a set $A \in \A$.  There are $\abs{A}!(n-\abs{A})!$ maximal chains containing $A$ and so
\begin{equation}
\label{eq1}
\sum_{A\in \A} \sum_{\C} w(A,\C) = \sum_{A\in \A}\frac{\abs{A}!(n-\abs{A})!}{c(A)}.
\end{equation}

Now, fix a maximal chain $\C$.  Suppose $\C$ contains $m$ sets from $\A$; then we have $c(A) \ge m$ for all such sets, since these $m$ sets themself form a chain of length $m$.  Thus, we have
\begin{equation}
\label{eq2}
\sum_{\C} \sum_{A\in \A} w(A,\C) \le \sum_{\C} \frac{m}{m} = n!.
\end{equation}
Comparing $\eqref{eq1}$ and $\eqref{eq2}$ yields
\begin{equation}
\label{eq3}
\sum_{A\in \A} \frac{1}{c(A) \binom{n}{\abs{A}}} \le 1,
\end{equation}
as desired.   Suppose we have equality in $\eqref{eq3}$.  Then the total weight on every chain must be $1$.   
% We claim that in this case $\A$ must equal a union of levels.  Suppose not.
Note that this trivially holds in the case where $\A$ equals a union of levels. We claim that this is the only case where equality holds. Suppose $\A$ is not a union of levels.
Then there is at least one level of $2^{[n]}$ containing both a set from $\A$ and one not from $\A$.  Moreover,  there must be such a pair with symmetric difference $2$.  Call these sets $A$ and $B$ where $A \in \A$ with $c(A)=m$ and $B \notin \A$.  Since $\abs{A}=\abs{B}=t$ and $A$ and $B$ have symmetric difference $2$, it follows that the intersection, $A\cap B$, is of size $t-1$ and the union, $A\cup B$, is of size $t+1$.

Consider any chain of the form  $\C_1 = \{\varnothing,A_1,A_2,\dots,A_{t-2},A\cap B, A,A\cup B, A_{t+2},\dots,[n]\}$, then the total weight along $\C_1$ must be $1$ in the equality case and so $\C_1$ contains $m$ sets from $\A$ each with weight $1/m$.    Now, consider the chain $\C_2 = \{\varnothing,A_1,A_2,\dots,A_{t-2},A\cap B, B,A\cup B, A_{t+2},\dots,[n]\}$. Since this chain contains the same sets as $\C_1$ except with $B \notin \A$ instead of $A \in \A$, the total weight of $\C_2$ is $1-1/m$, a contradiction.  
\end{proof}
\
\begin{proof}
We will use Katona's method of cyclic permutations~\cite{katona1972simple}.  By a cyclic permutation, $\sigma$, we mean a cyclic ordering $a_1 < a_2 < \dots < a_n < a_1$ of $[n]$.  A set $A \in \A$ is an interval in $\sigma$ if its elements are consecutive in $\sigma$.  An interval starts at $a_i$ if $a_i$ is the smallest element in the order on $A$ induced by $\sigma$.
We will double count pairs $(A,\sigma)$ where $A\in\A$ and $\sigma$ is a cyclic permutation of $[n]$ with the following weight function: 
\begin{displaymath}
w(A,\sigma)=
\begin{cases}
\frac{1}{m(A)}, & \text{if } A\in\A \text{ and } A \text{ is an interval in } \sigma,  \\
0, &\text{otherwise.}
\end{cases}
\end{displaymath}
Every set $A\in\A$ is an interval in exactly $r!(n-r)!$ cyclic permutations.  Thus, we have
\begin{displaymath}
\sum_{A\in\A}\sum_{\sigma} w(A,\sigma) = \sum_{A\in\A} \frac{r!(n-r)!}{m(A)}.
\end{displaymath}
Now we switch the order of summation and fix a cyclic permutation $\sigma$.   In Lemma~\ref{lemma:cyclemma}, whose proof is given below, we will show that the the total weight contributed by intervals along any cyclic permutation is at most $r$.  It follows that
\begin{displaymath}
\sum_{\sigma}\sum_{A\in\A}  w(A,\sigma) \le \sum_{\sigma}r = (n-1)!r.
\end{displaymath}
Dividing through by $(n-r)!r!$ gives the desired result.
\end{proof}

\begin{lemma} 
\label{lemma:cyclemma}
Let $\sigma$ be a cyclic permutation and denote by $\A^\sigma$ the collection of those sets in $\A$ which are intervals in $\sigma$.  Let $m(A)$ be defined as in the theorem, then
\begin{displaymath}
\sum_{A\in \A^\sigma} \frac{1}{m(A)} \le r.
\end{displaymath}
\end{lemma}

\begin{proof} Let $n=tr+s$ where $t$ and $s$ are integers and $0 \le s < r$.  Similarly, let $\abs{A^\sigma} = l r+k$ where $l$ and $k$ are integers and $0 \le k < r$.

First, we show that every set in $\A^\sigma$ is contained in an $l$-matching.  Notice that if an interval $A$ does not begin less than $r$ positions before or after an interval $B$, then $A$ and $B$ are disjoint.  Consider the cyclic order attained by removing all of those elements from $\sigma$ at which no $A \in \A^\sigma$ begins.  If two elements $a$ and $b$ are more than $r$ positions apart in this contracted order, then they must have been at least $r$ positions apart in the original order, and so the interval beginning at $a$ in $\sigma$ must be disjoint from the interval beginning at $b$.  Now take an arbitrary element of the contracted order and move along the order $r$ elements at a time.  After $l-1$ repetitions we are still a distance of at least $r$ from the starting element.  The starting element, along with the elements encountered at each of the $l-1$ steps constitute an $l$-matching.  Since our choice of starting element in the contracted order was arbitrary, it follows that every $A\in \A^\sigma$ is contained in an $l$-matching.

Now, we distinguish two cases based on whether $l=t$ or $l<t$.  If $l=t$, then since every $A\in\A^\sigma$ is contained in an $l$-matching we have,
\begin{displaymath}
\sum_{A\in\A^\sigma}\frac{1}{m(A)} = \frac{\abs{\A^\sigma}}{n/r} \le \frac{n}{n/r}=r.
\end{displaymath}

It remains to show the bound for the case $l<t$.  Since $l<t$ there are at least $r-k$ elements of $\sigma$ at which no set in $\A^{\sigma}$ begins. Choose any $r-k$ such elements.  

We will refer to these elements as ``fake'' starting positions as we will imagine additional intervals beginning at them.  The size of $\A^{\sigma}$ plus the number of fake starting positions is then $(l+1)r$. As before, we consider the contracted order containing only elements beginning sets in $\A^{\sigma}$ and the fake starting positions.  Since $(l+1)r$ is a multiple of $r$, we may partition the elements of the resulting order into $r$  classes each containing $l+1$ elements such that in any class the distance between two elements is a multiple of $r$. There are only $r-k$ fake starting positions, so it follows that some $k$ of these classes consist of only elements corresponding to sets in $\A^{\sigma}$. For each set $A$ contained in one of the $k$ classes consisting only of sets in $\A^{\sigma}$, we know that $m(A) \ge l+1$ since $A$ is in a $l+1$-matching with the other sets in the class.  It follows that the net weight contributed by such sets is at most 
\begin{equation}
\label{eqn:partone}
\frac{k(l+1)}{l+1}=k.
\end{equation}
Subtracting off the sets in $\A^{\sigma}$ which fill one of the classes and the fake starting positions leaves
\begin{displaymath}
r(l+1)-k(l+1)-(r-k)=l(r-k)
\end{displaymath}
sets.  Since every set in $\A^{\sigma}$ is contained in matching of size at least $l$, we have that the weight contribution of the remaining sets is at most
\begin{equation}
\label{eqn:parttwo}
\frac{l(r-k)}{l}=r-k.
\end{equation}
Thus, adding \eqref{eqn:partone} and \eqref{eqn:parttwo} we have that
\begin{displaymath}
\sum_{A\in\A^{\sigma}} \frac{1}{m(A)}\le k+(r-k)=r,
\end{displaymath}
as desired.
\end{proof}

\section{Results on perfect graphs, posets and sequences}\label{posets}

\begin{theorem} \label{superlemma}
Let $G$ be a perfect graph, and $f$ be a function on the set of pairs $(H, v)$ where $H$ is an induced subgraph of $G$ and $v\in V(H)$. Assume that $f$ satisfies $f(H,v) \le f(G[V(H)\cup\{w\}],v)$ for inducerd subgraphs $H$ of $G$ and $w\in V(G) \setminus V(H)$. Also assume that if $H$ is an independent set in $G$, then we have
\begin{equation*}
    \sum_{v\in V(H)}\frac{1}{f(H,v)} \le 1.
\end{equation*}
For an induced subgraph $H$ of $G$, let $c(H, v)$ be the size of the largest clique in $H$ that contains $v$. Then, we have
\begin{equation*}
    \sum_{v \in V(H)}\frac{1}{f(H, v)c(H, v)}\le 1.
\end{equation*}
\end{theorem}
\begin{proof}
We use induction on the size of the largest clique in $H$. If $H$ is an independent set, then we have $c(H, v)=1$ for all $v\in V(H)$ and the desired bound follows by assumption.
Otherwise, let $k$ be the size of the largest clique in $H$, and assume we have the desired bound whenever the maximum clique has size strictly less than $k$.
Define $B$ as the set of all vertices contained in a $k$-clique, that is $B=\{v\in V(H): c(H, v) = k\}$.
Since $G$ is perfect, we may conclude that the chromatic number of $H$ is also $k$. Let $R$ be any $k$-coloring of $H$. We form the natural partition of $B$ based on $R$ as
\begin{equation*}
    B_r=\{v \in B: v\text{ has color $r$ under }R\}.
\end{equation*}
Now, for any color $r\in R$, we know that every $k$-clique in $H$ must contain exactly one vertex with color $r$, and so the largest clique in $H \setminus B_r$ must have size exactly $k-1$. We apply induction to get that
\begin{equation}\label{eq:pg_sumup}
    \sum_{v\in V(H\setminus B_r)}\frac{1}{f(H\setminus B_r, v)c(H \setminus B_r, v)} \le 1
\end{equation}
for every $r\in R$.

Now, for all $v \in H$ and all $r \in R$, we have by assumption that $f(H, v) \ge f(H \setminus B_r, v)$ and trivially that $c(H, v) \ge c(H\setminus B_r,v)$. Recall that $B_r$ contains exactly one vertex from each $k$-clique with vertices in $B$ (which is every $k$-clique in $H$). Thus, for any $v \in B \setminus B_r$ it is clear that $c(H\setminus B_r,v)=k-1$, and so for all $v \in B$ we have $c(H \setminus B_r, v) = \frac{k-1}{k}c(H, v)$. Summing Equation~\eqref{eq:pg_sumup} over all $r$ gives us
\begin{align*}
    k
    &\ge \sum_{r \in R}\sum_{v\in V(H\setminus B_r)}\frac{1}{f(H\setminus B_r, v)c(H \setminus B_r, v)}\\
    &=\sum_{r\in R}\sum_{v\in V(H\setminus B)}\frac{1}{f(H\setminus B_r, v)c(H \setminus B_r, v)}+\sum_{r\in R}\sum_{v\in V(G[B\setminus B_r])}\frac{1}{f(H\setminus B_r, v)c(H \setminus B_r, v)}\\
    &\ge k \sum_{v\in V(H\setminus B)}\frac{1}{f(H, v)c(H, v)} + (k-1)\sum_{v\in V(G[B])}\frac{1}{f(H, v)\frac{k-1}{k}c(H, v)}\\
    &= k \sum_{v\in V(H)}\frac{1}{f(H, v)c(H, v)}.
\end{align*}
Dividing both sides by $k$ yields the desired inequality.
\end{proof}

The following corollary is a natural generalization of Theorem~\ref{localizedLYM} to any poset satisfying the LYM-inequality.
\begin{corollary} \label{posetcA}
Let $P$ be a ranked poset with $N_i$ elements of rank $i$. For $x\in P$, let $\abs{x}$ denote the rank of $x$.   Suppose that any antichain  $S \subset P$ satisfies the LYM-type inequality:
\begin{displaymath}
\sum_{x\in S} \frac{1}{N_{\abs{x}}}\le 1.
\end{displaymath}
Then, given any set $T \subset P$ we have
\begin{displaymath}
\sum_{x \in T} \frac{1}{N_{\abs{x}} c(T,x)} \le 1.
\end{displaymath}
\end{corollary}

\begin{proof}
Let $G$ be the comparability graph of $P$. A set $S \subset P$ corresponds to an induced subgraph~$G_S$ of~$G$. Observe that $S$ is an antichain if and only if $G_S$ is an independent set. Moreover the chains in the poset $P$ are in one-to-one correspondence with the cliques in $G$.  The assumption that an antichain $S$ in $P$ satisfies the LYM-inequality shows that the conditions of Theorem~\ref{superlemma} are satisfied when $f(G_S,x)$ is taken to be $N_{\abs{x}}$ for all $S$ and $x$. Then the conclusion of Corollary~\ref{posetcA} is immediate from the conclusion of Theorem~\ref{superlemma}.
\end{proof}

Now we deduce Theorem~\ref{localizedperfect} from Theorem~\ref{superlemma}.

\begin{proof}[Proof of Theorem~\ref{localizedperfect}]
Let $G$ be a perfect graph and for an induced subgraph $H$ and $x\in V(H)$ let $f(H,v)$ be the maximum size of an independent set in $H$ which contains $v$.  It is easy to see $f$ satisfies the condition increasing condition. If $H$ is an independent set then $f(H,v)=\abs{V(H)}$ and the other condition of   Theorem~\ref{superlemma} is satisfied.  Then taking $G=H$, the conclusion of Theorem~\ref{superlemma} gives the required bound:
\[
\sum_{v\in V(G)} \frac{1}{i(v)c(v)} = \sum_{v\in V(G)} \frac{1}{f(G,v)c(G,v)} \le 1.\qedhere
\]
\end{proof}
% Fix a poset $P$.  For all $S \subset P$ and $x \in S$ let $f(x,S) = A(x)$ (where $A(x)$ and $C(x)$ are defined in terms of $P$).  If $S$ is an antichain, then $f(x,S) = \abs{S}$ and the required condition of Theorem \ref{poset} hold.  Moreover it is clear that for $u \in P \setminus S$ that adding $u$ to $S$ can only increase the size of an antichain to which $x$ belongs.  Thus, we may apply Theorem %\ref{poset} and obtain that for any $S \subset P$
% \begin{displaymath}
% \sum_{x\in S} \frac{1}{A(x)C(x)} \le 1.
% \end{displaymath}
% Taking $S = P$ we obtain the form of Corollary \ref{dilworth}.

Again by considering the comparability graph the following corollary is immediate from Theorem~\ref{localizedperfect}. For any poset $P$ and $x \in P$ let $C(x)$ denote the size of the  largest chain $x$ belongs to and $A(x)$ the size of the largest antichain $x$ belongs to.
\begin{corollary}
\label{dilworth}
Let $P$ be any poset we have,
\begin{displaymath}
\sum_{x\in P} \frac{1}{A(x)C(x)}\le 1.
\end{displaymath}
\end{corollary}

By considering the natural poset defined on sequences of numbers $x_1,x_2,\dots,x_N$ where $x_i \prec x_j$ if $i<j$ and $x_i<x_j$ we obtain a localized version of the classical result of Erd\H{o}s and Szekeres~\cite{erdos1935combinatorial}. For a finite  sequence of real numbers let $i(x)$ and $d(x)$ denote the longest increasing and decreasing subsequence containing $x$, respectively. Then we have the following.
\begin{corollary}
For any finite sequence $S$ of real numbers
\[
\sum_{x\in S} \frac{1}{i(x)d(x)} \le 1.
\]
\end{corollary}

\section{Acknowledgements}
 The second author would like to thank Rutger Campbell and Tuan Tran for insightful discussions about this topic and Nika Salia for providing the example with the triangle and bow tie in the introduction.
 The research of the second author was supported by NKFIH grant K135800.
\bibliography{references.bib}

\begin{thebibliography}{10}

\bibitem{bollobas1965generalized}
B{\'e}la Bollob{\'a}s.
\newblock On generalized graphs.
\newblock {\em Acta Mathematica Hungarica}, 16(3-4):447--452, 1965.

\bibitem{bollobas2004multicoloured}
B{\'e}la Bollob{\'a}s, Peter Keevash, and Benny Sudakov.
\newblock Multicoloured extremal problems.
\newblock {\em Journal of Combinatorial Theory, Series A}, 107(2):295--312,
  2004.

\bibitem{borg2009short}
Peter Borg.
\newblock A short proof of a cross-intersection theorem of {Hilton}.
\newblock {\em Discrete mathematics}, 309(14):4750--4753, 2009.

\bibitem{bradac}
Domagoj Bradač.
\newblock A generalization of {Turán's} theorem.
\newblock {\em arXiv preprint arXiv:2205.08923}, 2022.

\bibitem{gallai1959maximal}
Paul Erd{\H{o}}s and Tibor Gallai.
\newblock On maximal paths and circuits of graphs.
\newblock {\em Acta Math. Acad. Sci. Hungar. v10}, pages 337--356, 1959.

\bibitem{erdos1961intersection}
Paul Erd{\H{o}}s, Chao Ko, and Richard Rado.
\newblock Intersection theorems for systems op finite sets.
\newblock {\em Quart. J. Math. Oxford Ser.(2)}, 12:313--320, 1961.

\bibitem{erdos1935combinatorial}
Paul Erd{\H{o}}s and George Szekeres.
\newblock A combinatorial problem in geometry.
\newblock {\em Compositio mathematica}, 2:463--470, 1935.

\bibitem{hilton1977intersection}
A.J.W. Hilton.
\newblock An intersection theorem for a collection of families of subsets of a
  finite set.
\newblock {\em Journal of the London Mathematical Society}, 2(3):369--376,
  1977.

\bibitem{katona}
Gyula~O.H. Katona.
\newblock Sperner type theorems, {Ph. D.} {T}hesis, {Eötvös Loránd
  University}.
\newblock 1968.

\bibitem{katona1972simple}
Gyula~O.H. Katona.
\newblock A simple proof of the {Erd{\H{o}}s-Chao Ko-Rado theorem}.
\newblock {\em Journal of Combinatorial Theory, Series B}, 13(2):183--184,
  1972.

\bibitem{keevash2004multicolour}
Peter Keevash, Mike Saks, Benny Sudakov, and Jacques Verstra{\"e}te.
\newblock Multicolour {T}ur{\'a}n problems.
\newblock {\em Advances in Applied Mathematics}, 33(2):238--262, 2004.

\bibitem{lubell1966short}
David Lubell.
\newblock A short proof of {Sperner's} lemma.
\newblock {\em Journal of Combinatorial Theory}, 1(2):299, 1966.

\bibitem{meshalkin1963generalization}
Lev~D. Meshalkin.
\newblock Generalization of {Sperner’s} theorem on the number of subsets of a
  finite set.
\newblock {\em Theory of Probability \& Its Applications}, 8(2):203--204, 1963.

\bibitem{mirsky1971dual}
Leon Mirsky.
\newblock A dual of {Dilworth's} decomposition theorem.
\newblock {\em The American Mathematical Monthly}, 78(8):876--877, 1971.

\bibitem{motzkin1965maxima}
Theodore~S. Motzkin and Ernst~G. Straus.
\newblock Maxima for graphs and a new proof of a theorem of tur{\'a}n.
\newblock {\em Canadian Journal of Mathematics}, 17:533--540, 1965.

\bibitem{sperner1928satz}
Emanuel Sperner.
\newblock Ein satz {\"u}ber untermengen einer endlichen menge.
\newblock {\em Mathematische Zeitschrift}, 27(1):544--548, 1928.

\bibitem{yamamoto1954logarithmic}
Koichi Yamamoto.
\newblock Logarithmic order of free distributive lattice.
\newblock {\em Journal of the Mathematical Society of Japan}, 6(3-4):343--353,
  1954.

\end{thebibliography}

\end{document}